\documentclass[11pt]{amsart}  
\usepackage[margin=1in]{geometry}

\usepackage{MnSymbol}
\usepackage{colonequals}

\usepackage{hyperref}
\hypersetup{
pdflang={en-US},
pdftitle={Normal bundles of rational curves in Grassmannians},
pdfauthor={Izzet Coskun, Eric Larson, and Isabel Vogt},
pdfsubject={Algebraic Geometry},
pdfkeywords={Grassmannians, rational curves, normal bundles}
}

\makeatletter
\newcommand{\oset}[3][0ex]{%
  \mathrel{\mathop{#3}\limits^{
    \vbox to#1{\kern-2\ex@
    \hbox{\(\scriptstyle#2\)}\vss}}}}
\makeatother

\newcommand{\pp}{\mathbb{P}}
\newcommand{\zz}{\mathbb{Z}}
\renewcommand{\O}{\mathcal{O}}
\newcommand{\Hom}{\operatorname{Hom}}
\newcommand{\NCtop}{N_{C \to p}}
\newcommand{\NCtoh}{N^{C\to h}}
\newcommand{\Gr}{\operatorname{Gr}}
\newcommand{\arrowdown}{\rcurvearrowright}
\newcommand{\arrowup}{\curvearrowright}
\newcommand{\defi}[1]{\textsf{#1}} 

\newtheorem{thm}{Theorem}[section]
\newtheorem{lem}[thm]{Lemma}
\newtheorem{prop}[thm]{Proposition}
\newtheorem{cor}[thm]{Corollary}
\newtheorem{conj}[thm]{Conjecture}

\theoremstyle{definition}
\newtheorem{defin}[thm]{Definition}

\theoremstyle{remark}
\newtheorem{rem}[thm]{Remark}

\title{Normal bundles of rational curves in Grassmannians}

\author{Izzet Coskun}
\address{Department of Mathematics, Statistics, and CS \\
University of Illinois at Chicago, Chicago IL 60607}
\email{icoskun@uic.edu}
\author{Eric Larson}
\address{Department of Mathematics, Brown University}
\email{elarson3@gmail.com}
\author{Isabel Vogt}
\address{Department of Mathematics, Brown University}
\email{ivogt.math@gmail.com}

\thanks{During the preparation of this article, I.C.\ was supported
by  NSF grant DMS-2200684,  E.L. was supported by
NSF  grant DMS-2200641, and I.V. was supported by NSF grant DMS-2200655.}

\begin{document}

\begin{abstract}
    In projective space over fields of characteristic different from \(2\), the normal bundle of a general nondegenerate rational curve is balanced.  The corresponding statement for rational curves in other Grassmannians can fail.  Nevertheless, we prove that the normal bundle of a general rational curve in a Grassmannian decomposes into a direct sum of line bundles whose degrees are at most 2 apart.
\end{abstract}

\maketitle

\section{Introduction}

Rational curves play a central role in the birational geometry and the arithmetic of varieties. Given a rational curve $C$ in a smooth projective variety $X$, the normal bundle $N_{C/X}$ controls the deformations of $C$ in $X$. Consequently, understanding the properties of the normal bundle of $C$ is a fundamental question in algebraic geometry. In this paper, we study the normal bundles of general rational curves in  Grassmannians. We work over an algebraically closed field of arbitrary characteristic.

A vector bundle $W$ of rank $r$ on $\pp^1$ decomposes as a direct sum of line bundles $$W \cong \bigoplus_{i=1}^r \O_{\pp^1}(a_i)$$ for a unique sequence of integers $a_1 \geq \cdots \geq a_r$. We say that the vector bundle is \defi{$j$-balanced} if $a_1 - a_r \leq j$. If $j=0$, the bundle is called \defi{perfectly balanced}. If $j=1$, the bundle is called \defi{balanced}.

Let $\Gr(a, a+b)$ denote the Grassmannian parameterizing $a$-dimensional subspaces of an $(a+b)$-dimensional vector space $V$. 
The space of rational curves of degree $d$ in $\Gr(a, a+b)$ is irreducible \cite{kimpandharipande}. Hence, it makes sense to ask for the splitting type of the normal bundle $N_{C/\Gr(a,a+b)}$ of a general rational curve $C$ of degree $d$ in $\Gr(a,a+b)$.

In projective space, normal bundles of rational curves have been studied extensively (see for example \cite{AlzatiRe, aly, coskunriedl, eisenbudvandeven, eisenbudvandeven82, ghionesacchiero, ran, sacchiero80, sacchiero82, interpolation}). In particular, if $C$ is a general nondegenerate rational curve in $\pp^b$, then $N_{C/\pp^b}$ is balanced \cite{sacchiero80, interpolation} except when the base field has characteristic \(2\). 
For Grassmannians other than projective spaces, only partial results are known.
For example, Ziv Ran proved that when $a=2$ or when $a=b$, there exist infinitely many degrees $d$ for which the normal bundle of the general rational curve is balanced \cite{ran2}. Our main theorem in this paper is the following.

\begin{thm}\label{thm-main}
Let $C$ be a general rational curve of degree $d$ in the Grassmannian $\Gr(a, a+b)$. Then  $N_{C/\Gr(a,a+b)}$ is $2$-balanced.
\end{thm}

There are at least three different families of cases in which Theorem~\ref{thm-main} is sharp, i.e., for which $N_{C/\Gr(a,a+b)}$ is \emph{not} balanced.

\subsection{The degeneracy exceptions}
Without loss of generality, suppose \(a \leq b\). If \(1 < d < a\), then
\[N_{C/\Gr(a, a + b)} \simeq N_{C/\Gr(d, d + b)} \oplus Q|_C^{a - d},\]
where \(Q\) is the tautological quotient bundle.
This cannot be balanced because
\[\mu(Q|_C) = \frac{d}{b} < 1 < \frac{(b + d)d - 2}{bd - 1} = \mu(N_{C/\Gr(d, d + b)}).\]
Similarly, if \(a < d < b\), then
\[N_{C/\Gr(a, a + b)} \simeq N_{C/\Gr(a, a+d)} \oplus (S|_C^\vee)^{b - d},\]
where \(S\) is the tautological subbundle.
This cannot be balanced because
\[\mu(S|_C^\vee) = \frac{d}{a} < \frac{(a + d)d - 2}{ad - 1} - 1 = \mu(N_{C/\Gr(a, a + d)}) - 1.\]

In particular, when \(a = 1\), this argument recovers that the normal bundle of a rational curve of degree \(d\) in a projective space \(\pp^b\) cannot be balanced for \(1 < d < b\).

\subsection{\boldmath The characteristic \(2\) exceptions\label{ss:c2}}
When \(a = 1\), we have an Euler sequence
\[0 \to N_{C/\Gr(1, b + 1)}^\vee (1) \to \O_{\pp^1}^{b + 1} \to \mathcal{P}^1(\O_{\pp^1}(d)) \to 0,\]
where \(\mathcal{P}^1(\O_{\pp^1}(d))\) is the bundle of principal parts of \(\O_{\pp^1}(d)\).
If, furthermore, the characteristic is \(2\), then \(\mathcal{P}^1(\O_{\pp^1}(d)) \simeq F^* F_* \O_{\pp^1}(d)\), where \(F \colon \pp^1 \to \pp^1\) denotes the Frobenius morphism. Hence, \(N_{C/\Gr(1, b + 1)}^\vee (1)\) is the pullback of a vector bundle under Frobenius, and so all of its summands are even (see \cite[\S 3]{clv} and \cite[\S 2.2]{interpolation}).
We conclude that \(N_{C / \Gr(1, b + 1)}\) can only be balanced if
\[\frac{2 - 2d}{b - 1} = \mu(N_{C/\Gr(1, b + 1)}^\vee (1)) \in 2\zz,\]
or upon rearrangement if \(d \equiv 1\) mod \(b - 1\).

\subsection{The tangent bundle splitting exceptions}
Ramella proved that the restriction of $T_{\pp^b}$ to a general rational curve is balanced \cite{ramella}. The analogous statement for Grassmannians is not always true. 
The tangent bundle of the Grassmannian has the structure of a tensor product $$T_{\Gr(a, a+b)} \cong S^\vee \otimes Q.$$
The restrictions of the bundles $S^\vee$ and $Q$ to a general rational curve are balanced \cite{Ma19}. However, when $a, b >1$, the tensor structure of $T_{\Gr(a, a+b)}$ can obstruct  $T_{\Gr(a, a+b)}|_C$  from being balanced. The restricted tangent bundle is balanced if  and only if $a$ or $b$ divides the degree $d$ of $C$. Otherwise, the restricted tangent bundle is 
only $2$-balanced. The possible splitting types of $T_{\Gr(a,a+b)}|_C$ have been studied in \cite{Ma19}.

The fact that $T_{\Gr(a,a+b)}|_C$ is not balanced under certain numerical conditions obstructs the normal bundle  $N_{C/\Gr(a,a+b)}$ from being balanced. Consider the normal bundle exact sequence

\begin{equation}\label{seq-normalbundle}
    0 \to T_C  \to T_{\Gr(a, a+b)}|_C \to N_{C/\Gr(a, a+b)} \to 0.
\end{equation}
Using the division algorithm, write 
$$d= aq_1 + r_1 \ \ \mbox{and} \ \ d= bq_2 + r_2$$ with $0 \leq r_1 < a$ and $0 \leq r_2 < b$. If $C$  is a general rational curve of degree $d$ in $\Gr(a, a+b)$, then the splitting type of $T_{\Gr(a, a+b)}|_C$ is
$$\O_{\pp^1}(q_1 + q_2 + 2)^{r_1 r_2} \oplus \O_{\pp^1} (q_1 + q_2 +1)^{r_1 (b-r_2)+ r_2(a-r_1)} \oplus \O_{\pp^1}(q_1+q_2)^{(a-r_1)(b-r_2)}.$$
If $q_1+ q_2-1 < (a-r_1)(b-r_2)$, then twisting the sequence (\ref{seq-normalbundle}) by $\O_{\pp^1}(-q_1-q_2-2)$ and comparing the first cohomology groups, we see that $N_{C/\Gr(a, a+b)}$ must have a summand equal to  $\O_{\pp^1}(q_1+q_2)$. On the other hand, if $r_1 r_2 \not= 0$, then $N_{C/\Gr(a,a+b)}$ must also have a summand equal to $\O_{\pp^1}(k)$ with $k \geq q_1 + q_2 + 2$. Hence, the normal bundle cannot be balanced. 

\subsection{Conjecture}
We conjecture that these three obstructions account for all cases in which $N_{C/\Gr(a, a+b)}$ is not balanced.

\begin{conj}\label{conj-normal}
Let $C$ be a general rational curve of degree $d$ in $\Gr(a, a+b)$, where without loss of generality we take \(a \leq b\).
Then $N_{C/\Gr(a, a+b)}$ is balanced if and only if none of the following occur:
\begin{itemize}
\item The degeneracy exceptions: We have \(d < b\) and \(d \notin \{1, a\}\).
\item The characteristic \(2\) exceptions: The characteristic is 2 and \(a = 1\) and \(d \not\equiv 1\) mod \(b - 1\).
\item The tangent bundle splitting exceptions: We have \(r_1 r_2 \neq 0\) and $q_1+ q_2 \leq (a-r_1)(b-r_2)$.
\end{itemize}
\end{conj}

\subsection*{Strategy} We will prove Theorem \ref{thm-main} by specializing rational curves to reducible nodal rational curves and using exact sequences coming from projection. In \S \ref{sec-prelim}, we give general results on using exact sequences to deduce $2$-balancedness of vector bundles on $\pp^1$. In \S \ref{sec:generalized_pointing}, we discuss natural subbundles of the normal bundle and how they relate to exact sequences obtained by projection. In \S \ref{sec:one_sec_degen}, we  study the normal bundles of curves under one-secant degenerations using  modifications of vector bundles. Finally, in \S \ref{sec-proof}, we give a combinatorial argument to show that these ingredients can be combined to prove Theorem \ref{thm-main}.

\subsection*{Acknowledgments} We would like to thank Atanas Atanasov, Lawrence Ein, Gavril Farkas, Joe Harris, Eric Jovinelly, Carl Lian, Eric Riedl, Geoffrey Smith and David Yang for invaluable conversations about the normal bundles of curves. 

\section{Preliminaries}\label{sec-prelim}

\subsection{Notational Conventions} 
Let $\Gr(a, a+b)$ denote the Grassmannian of $a$-dimensional subspaces of an $(a+b)$-dimensional vector space. We will sometimes interpret this Grassmannian as the Grassmannian parameterizing linear spaces $\pp^{a-1}$ in $\pp^{a+b-1}$. We use upper case letters such as \(C\), \(\Gamma\), and \(\Omega\), to denote subvarieties of $\Gr(a, a+b)$.
We use lower case letters for linear subspaces of \(\pp^{a+b-1}\), including \(p \in \pp^{a+b-1}\) for a point and \(h \subset \pp^{a+b-1}\) for a hyperplane.

When the ambient space \(X\) is clear, we write \(N_C\) for the normal bundle \(N_{C/X}\) of \(C\) in \(X\).

\subsection{The tautological subbundle and quotient bundle}

Let \(V\) be an \((a+b)\)-dimensional vector space.  The Grassmannian \(\Gr(a, V)\) has two tautological vector bundles: the \defi{tautological subbundle} \(S\) of rank \(a\), whose fiber over a point \([\lambda] \in \Gr(a, V)\) is the subspace \(\lambda \subset V\); the \defi{tautological quotient bundle} \(Q\) of rank \(b\), whose fiber over a point \([\lambda] \in \Gr(a, V)\) is the quotient \( V/\lambda\).  These bundles fit together into a exact sequence
\[0 \to S \to V \otimes \O_{\Gr(a, V)} \to Q \to 0.\]
If \(C \subset \Gr(a, V)\) is a curve of degree \(d\), then \(S|^\vee_C\) and \(Q|_C\) are vector bundles of degree \(d\) on \(C\).

\subsection{Duality} Let \(V\) be a vector space of dimension \(a + b\).  Write \(V^\vee\) for the dual vector space (also of dimension \(a+b\)).  There is a natural isomorphism \(\Gr(a, V) \simeq \Gr(b, V^\vee)\) sending an \(a\)-dimensional subspace \(\lambda \subset V\) to its annihilator \(\lambda^0 \subset V^\vee\).  This isomorphism exchanges \(S^\vee\) and \(Q\).
We recall that if \(W_1, W_2 \subset V\) are subspaces, then 
\begin{equation}\label{eq:dual_facts}
    \left(W_1 + W_2\right)^0 = W_1^0 \cap W_2^0 \subset V^\vee \qquad \text{and} \qquad (W_1 \cap W_2)^0 = W_1^0 + W_2^0 \subset V^\vee.
\end{equation}

\subsection{\boldmath \(2\)-balanced bundles and exact sequences}
The basic technique that we will use to prove \(2\)-balancedness of a given bundle is to position it in an exact sequence between two bundles that are themselves \(2\)-balanced.  This alone is insufficient to guarantee \(2\)-balancedness: the degrees of the sub and quotient \(2\)-balanced bundles must be drawn from the \textit{same} interval of length \(2\).

\begin{lem}\label{lem:interval}
    Let
    \(0 \to E \to F \to G \to 0\)
    be an exact sequence of vector bundles on \(\pp^1\).  
    \begin{enumerate}
        \item \label{at_most} If all summands of \(E\) and \(G\) have degree at most \(n\), then so do all summands of \(F\).
        \item \label{at_least} If all summands of \(E\) and \(G\) have degree at least \(m\), then so do all summands of \(F\).
    \end{enumerate}
    In particular, if all of the summands of \(E\) and \(G\) have degree that lie in a fixed interval, then the degrees of all of the summands of \(F\) lie in the same interval.
\end{lem}
\begin{proof}
    By assumption~\eqref{at_most}, we have \(h^0(E(-n-1)) = h^0(G(-n-1)) = 0\), so \(h^0(F(-n-1)) = 0\) as well.  Hence every summand of \(F\) has degree at most \(n\).
    
    By assumption~\eqref{at_least}, we have \(h^1(E(-m-1)) = h^1(G(-m-1)) = 0\), so \(h^1(F(-m-1)) = 0\) as well.  Hence every summand of \(F\) has degree at least \(m\).
\end{proof}

\begin{lem}\label{lem:EG}
Let \(E\) and \(G\) be vector bundles on \(\pp^1\) and suppose that \(E\) is balanced and \(G\) is \(2\)-balanced.  If \(\mu(E) \leq \mu(G)\), then each summand of \(G\) has degree at least \(\lfloor \mu(E) \rfloor - 1\).
\end{lem}
\begin{proof}
    Write \(E \simeq \O(\ell)^{\oplus p} \oplus \O(\ell + 1)^{\oplus q}\), and \(G \simeq \O(n)^{\oplus x} \oplus\O(n+1)^{\oplus y} \oplus \O(n+2)^{\oplus z}\).  Assume that \(p \neq 0\) so that \(\lfloor \mu(E) \rfloor = \ell\). Assume to the contrary that \(n \leq \ell - 2\) and \(x \neq 0\).  Then by our assumption
\[\ell + \frac{q}{p+q} = \mu(E) \leq \mu(G) = n + 2 - \frac{2x + y}{x + y + z} \leq \ell - \frac{2x+y}{x + y + z},\]
which is impossible since \(q \geq 0\) and \(x > 0\).
\end{proof}

\begin{lem}\label{lem:ses_2bal}
    Let
    \(0 \to E \to F \to G \to 0\)
    be an exact sequence of vector bundles on \(\pp^1\).  Suppose that \(E\) is balanced, and that \(G\) is \(2\)-balanced.
    \begin{enumerate}
    \item\label{subsmaller} If
    \(\mu(E) \leq \mu(F)\),
    then \(F\) is \(2\)-balanced or each summand of \(G\) has degree at least 
    \(\lfloor \mu(E) \rfloor\).
        \item\label{subbigger} If
    \(\mu(F) \leq \mu(E)\),
    then \(F\) is \(2\)-balanced or each summand of \(G\) has degree at most \(\lceil \mu(E) \rceil\).
    \end{enumerate}
\end{lem}
\begin{proof}
Since \(\mu(F)\) is the weighted average of \(\mu(E)\) and \(\mu(G)\), the assumption \(\mu(E) \leq \mu(F)\) (respectively, \(\mu(E) \geq \mu(F)\)) is equivalent to \(\mu(E) \leq \mu(G)\) (respectively, \(\mu(E) \geq \mu(G)\)). 
In case~\eqref{subsmaller}, by Lemma~\ref{lem:EG} applied to \(E\) and \(G\), we have that every summand of \(G\) has degree at least \(\lfloor \mu(E) \rfloor - 1\).  Either we are done, or \(G\) contains a summand of degree exactly
\(\lfloor \mu(E) \rfloor - 1\). In the second case, since \(G\) is \(2\)-balanced, the degrees of all of its summands lie in the interval \(\left[\lfloor \mu(E) \rfloor - 1, \lfloor \mu(E) \rfloor +1 \right]\).  Since the same is true of \(E\), we conclude that the same must be true of \(F\), and consequently it is \(2\)-balanced, by Lemma~\ref{lem:interval}.

In case~\eqref{subbigger}, by Lemma~\ref{lem:EG} applied to \(E^\vee\) and \(G^\vee\), every summand of \(G\) has degree at most \(\lceil \mu(E) \rceil + 1\).  Either we are done or \(G\) has a summand of degree exactly \(\lceil \mu(E) \rceil + 1\); as above, this implies that the degrees of all of the summands of both \(G\) and \(E\), and hence \(F\) by Lemma~\ref{lem:interval}, lie in the interval \(\left[ \lceil \mu(E) \rceil - 1, \lceil \mu(E) \rceil + 1\right]\).  Thus \(F\) is \(2\)-balanced. \qedhere
\end{proof}

\subsection{Modifications of Vector Bundles}\label{sec:mods}

Here we briefly recall notation and basic properties of modifications of vector bundles. For a more detailed discussion, see \cite[\S 2--4]{aly}.

Let \(X\) be a scheme and \(D \subset X\) be a Cartier divisor. If \(E\) is a vector bundle on \(X\) and \(F\) is a subbundle, then we define the \defi{negative modification}
\[E[D \oset{-}{\to} F] = \ker\left(E \to (E/F)|_D\right),\]
and the positive modification
\[E[D \to F] = E[D \oset{-}{\to} F](D).\]
(This differs slightly from the notation in \cite{aly}, where negative modifications are denoted with the notation \(E[D \to F]\) and no separate notation is given for positive modifications.)

More generally, if \(D_1, D_2, \ldots, D_n\) are disjoint Cartier divisors and \(F_1, F_2, \ldots F_n\) are subbundles, then we can similarly define the \defi{multiple modification}
\[E[D_1 \to F_1] [D_2 \to F_2] \cdots [D_n \to F_n].\]
More care is needed when the supports of the \(D_i\) intersect.  In this paper we will only need to deal with the case that \(D = D_1 = D_2\) and \(F_1 \cap F_2\) is flat over \(X\).  In this case, \(F_2\) is a subbundle of \(E[D \to F_1]\) away from \(D\), which extends uniquely to a subbundle \(F_2'\) over all of \(X\).  The multiple modification is then defined by
\[E[D \to F_1][D \to F_2] \colonequals E[D \to F_1][D \to F_2'].\]

In nice cases, short exact sequences of vector bundles induce short exact sequences of modifications. To make this precise in the only case that will appear for us, suppose
\[0 \to E \to F \to G \to 0\]
is a short exact sequence of vector bundles and \(S \subset F\) is a subbundle so that \(S \cap E\) is flat over \(X\). Then we obtain a short exact sequence
\begin{equation}\label{eq:ses_mod}
    0 \to E[D \to S \cap E] \to F[D \to S] \to G[D \to S/(S \cap E)] \to 0.
\end{equation}

\section{Generalized pointing bundles}\label{sec:generalized_pointing}

\subsection{\boldmath Projection from \(p\) and intersection with \(h\)}

Let \(V\) be an \((a + b)\)-dimensional vector space and consider the Grassmannian \(\Gr(a, V)\).  Let \(p \in \pp V\) be a general point.  Let \(U_p \subset \Gr(a, V)\) be the open locus of \(a\)-planes \(\lambda\) such that \(p \notin \lambda\); the complement of \(U_p\) is a Schubert variety of codimension \(b\).
The point \(p\) defines a section of the universal quotient bundle \(Q|_{U_p}\) on \(U_p\) up to scaling given by \(\lambda \mapsto p \in \pp (V/\lambda)\).  In other words, the point \(p\) defines a subbundle of \(Q|_{U_p}\) that is isomorphic to \(\O_{U_p}\), but this choice of isomorphism depends on a scalar (the choice of a nonzero vector in the one-dimensional subspace corresponding to \(p\)).  By slight abuse of notation, we will denote this subbundle by \(\O \otimes p\).

On \(U_p \subset \Gr(a, V)\) there is a projection map
\(\pi_p \colon U_p  \to \Gr(a, V/p)\), sending an \(a\)-dimensional subspace of \(V\) to its image in \(V/p\).  Let \(T_{\pi_p} \subset T_{\Gr(a, V)}|_{U_p}\) be the vertical tangent bundle of this map:
\[T_{\pi_p} \colonequals \ker\left(T_{\Gr(a, V)}|_{U_p} \to T_{\Gr(a, V/p)} \right). \]
A tangent vector to \(U_p\) is contained in \(T_{\pi_p}\) if the corresponding deformation of the \(a\)-dimensional subspace \(\lambda\) is contained in the \((a + 1)\)-dimensional subspace \(\lambda + p\).  In terms of the identification \(T_{\Gr(a, V)} \simeq S^\vee \otimes Q\), the fiber of \(T_{\pi_p} \subset T_{\Gr(a, V)}\) at \(\lambda\) is
\[\Hom\left(\lambda, (\lambda + p)/\lambda\right) \subset \Hom\left(\lambda, V/\lambda\right).\]
Since the identification \(T_{\Gr(a, V)} \simeq S^\vee \otimes Q\) is functorial in \(V\), we obtain the following.

\begin{lem}\label{lem:tangent_poining_Sdual}
    Under the natural isomorphism \(T_{\Gr(a, V)} \simeq S^\vee \otimes Q\), we have
    \(T_{\pi_p} = (S^\vee \otimes p)|_{U_p}\).
\end{lem}

We can make a similar construction using a codimension \(1\) subspace \(h \subset V\) (the fact that it is the same follows by exchanging \(V\) for \(V^\vee\) and the duality relations in \eqref{eq:dual_facts}).  Let \(U^h\) be the locus of \(a\)-planes \(\lambda\) in \(\Gr(a, V)\) such that \(\lambda \not\subset h\).
The hyperplane \(h\) defines a subbundle of \(S^\vee|_{U^h}\) isomorphic to \(\O_{U^h}\) (sending \(\lambda\) to the linear functional \(\lambda \to \lambda/(\lambda\cap h) = V/h \simeq k\), where the isomorphism with \(k\) is the choice of scalar as above); we denote it  by \(h \otimes \O\).  Let \(\pi^h \colon U^h \to \Gr(a-1, h)\) denote the map sending \(\lambda\) to \(\lambda \cap h\).  Let \(T_{\pi^h} \colonequals \ker\left( T_{\Gr(a, V)}|_{U^h} \to T_{\Gr(a-1, h)}\right)\) be the vertical tangent bundle.  In terms of the identification \(T_{\Gr(a, V)} \simeq S^\vee \otimes Q\), the fiber of \(T_{\pi^h} \subset T_{\Gr(a, V)}\) at \(\lambda\) is
\[\Hom\left(\lambda/(\lambda \cap h), V/\lambda\right) \subset \Hom\left(\lambda, V/\lambda\right).\]  Applying duality, Lemma~\ref{lem:tangent_poining_Sdual} implies the following.

\begin{lem}\label{lem:tangent_poining_Q}
    Under the natural isomorphism \(T_{\Gr(a, V)} \simeq S^\vee \otimes Q\), we have
    \(T_{\pi^h} = (h \otimes Q)|_{U^h}\).
\end{lem}

\begin{lem} \label{lem-pinth}
On \(U_p \cap U^h\), we have \(T_{\pi_p} \cap T_{\pi^h} = h \otimes p \simeq \O_{U_p \cap U^h}\).
\end{lem}
\begin{proof}
In terms of the identification \(T_{\Gr(a, V)} \simeq S^\vee \otimes Q\), the fiber of \(T_{\pi_p} \cap T_{\pi^h} \subset T_{\Gr(a, V)}\) at \(\lambda\) is
\[\Hom(\lambda / (\lambda \cap h), (\lambda + p)/\lambda) \simeq \Hom(V/h, p). \qedhere\]
\end{proof}
\subsection{Restricting to general curves}

When \(b\geq 2\), respectively \(a \geq 2\), a general (rational) curve \(C\) of fixed degree in \(\Gr(a, V)\) will be contained in \(U_p\), respectively \(U^h\), for a general choice of \(p\), respectively \(h\). We may therefore restrict the bundles \(T_{\pi_p}\) and \(T_{\pi^h}\) to \(C\).  The following shows that this plays well with the quotient map to the normal bundle (i.e., these bundles do not meet \(T_C \subset T_{\Gr(a, V)}|_C\)).

\begin{lem}\hfill 
\begin{enumerate}
    \item\label{part:p} Let \(C \subset U_p\) be general.  The map \(T_{\pi_p}|_C \to N_C\) is an injection of vector bundles.
    \item Let \(C \subset U^h\) be general.  The map \(T_{\pi^h}|_C \to N_C\) is an injection of vector bundles.
\end{enumerate}
    
\end{lem}
\begin{proof}
    By duality it suffices to prove part \eqref{part:p}.  This follows because \(\pi_p|_C \colon C \to \Gr(a, V/p)\) is a general rational curve of degree \(d\) in \(\Gr(a, a + b - 1)\).  Hence this map is unramified.
\end{proof}

\begin{defin}\hfill
\begin{enumerate}
    \item The \defi{lower generalized pointing bundle} \(\NCtop\) is the image of \(T_{\pi_p}|_C\) in \(N_C\).  Equivalently, \\ \(\NCtop =  \ker\left(N_C \to N_{\pi_p(C)}\right)\), or further equivalently, this is the normal bundle of \(C\) in 
        \[ \Gamma_p \colonequals \left\{\gamma \in \Gr(a, a+b) : \gamma \subset c + p \text{ for some } c \in C\right\}.\]
    \item     The \defi{upper generalized pointing bundle} \(\NCtoh\) is the image of \(T_{\pi^h}|_C\) in \(N_C\).  Equivalently, \\ \(\NCtoh =  \ker\left(N_C \to N_{\pi^h(C)}\right)\), or further equivalently, this is the normal bundle of \(C\) in 
        \[ \Omega^h \colonequals \left\{\gamma \in \Gr(a, a+b) :  c\cap h \subset \gamma \text{ for some } c \in C\right\}.\]
    \item The intersection \(\NCtop \cap \NCtoh\) is therefore the normal bundle of \(C\) in
    \[\Sigma^h_p \colonequals \left\{\gamma \in \Gr(a, a+b) :  c\cap h \subset \gamma \subset c + p \text{ for some } c \in C\right\}\subset  \Gamma_p \cap \Omega^h.\]
\end{enumerate}
\end{defin}

By construction, these bundles lie in short exact sequences corresponding to projection \(\pi_p\) from \(p\) or intersection \(\pi^h\) with \(h\):
\begin{gather}
    0 \to \NCtop \to N_C \to (\pi_p)^*N_{\pi_p(C)}  \to 0 \label{eq:proj_p} \\
    0 \to \NCtoh \to N_C \to (\pi^h)^*N_{\pi^h(C)}  \to 0. \label{eq:int_h}
\end{gather}
From the global results Lemmas~\ref{lem:tangent_poining_Sdual}, \ref{lem:tangent_poining_Q}, and \ref{lem-pinth} we immediately obtain the following.

\begin{cor}\label{cor:isom_SQ}\hfill
\begin{enumerate}
    \item\label{isomS} Suppose \(C \subset U_p\).  Then \(\NCtop \simeq S|^\vee_C\).
    \item\label{isomQ} Suppose \(C \subset U^h\).  Then \(\NCtoh \simeq Q|_C\).
    \item\label{isomO}
    Suppose \(C \subset U_p \cap U^h\). Then \(\NCtop \cap \NCtoh \simeq \O_C\).
\end{enumerate}
\end{cor}

\begin{cor}\label{lem:gen_mods_S}
Fix \(x \in C\) and \(p\). For general \(h\), the fiber \((N_{C \to p} \cap N^{C \to h})|_x\) is general in \(N_{C \to p}|_x\).
\end{cor}

\subsection{Modifications towards generalized pointing bundles}
Recall the notation from Section~\ref{sec:mods} of modifications of vector bundles along divisors.
Of particular interest to us will be the case when \(X = C\) is a curve in \(\Gr(a, a + b)\), and \(E = N_C\) is its normal bundle, and \(F\) is one of \(N_{C \to p}\) or \(N^{C \to h}\). Since this case will occur so often, we introduce notation for modifications towards these subbundles:
\[N_C[x \arrowdown p] \colonequals N_C[x \to N_{C \to p}] \quad \text{and} \quad N_C[x \arrowup h] \colonequals N_C[x \to N^{C \to h}].\]

\section{One-secant degeneration}\label{sec:one_sec_degen}

Let \(a, b \geq 2\), and let \(L\) be a line in \(\Gr = \Gr(a, a + b) = \Gr(a, V)\).
Equivalently, let \(\omega \subset \gamma\) be an \((a - 1)\)-plane contained in an \((a + 1)\)-plane;
the line \(L\) then consists of all \(a\)-planes \(\ell\) with \(\omega \subset \ell \subset \gamma\).
Observe that \(S|_L\) contains a trivial subbundle of rank \(a - 1\) corresponding to \(\omega\),
and that \(Q|_L\) surjects onto a trivial quotient of rank \(b - 1\) corresponding to \(V / \gamma\).
Thus
\[T_{\Gr}|_L \simeq S^\vee \otimes Q \simeq [\O_{\pp^1}(1) \oplus \O_{\pp^1}^{a - 1}  ] \otimes [\O_{\pp^1}(1) \oplus \O_{\pp^1}^{b - 1} ],\]
and so
\[N_{L / \Gr} \simeq \O_{\pp^1}(1)^{(a - 1) + (b - 1)} \oplus \O_{\pp^1}^{(a - 1)(b - 1)}.\]
The positive subbundle of \(N_{L / \Gr}\) can be described naturally in terms of \(\omega\) and \(\gamma\):
Define
\begin{align*}
\Omega(\omega) &= \{\text{\(a\)-planes containing \(\omega\)}\} \\
\Gamma(\gamma) &= \{\text{\(a\)-planes contained in \(\gamma\)}\}.
\end{align*}
Then $\Omega(\omega) \cong \pp^b$, and $\Gamma(\gamma) \cong \pp^a$, and 
\[N_{L / \Omega(\omega)} \simeq \O(1)^{b - 1} \quad \text{and} \quad N_{L / \Gamma(\gamma)} \simeq \O(1)^{a - 1},\]
and the direct sum of these two bundles is the positive subbundle of \(N_{L / \Gr}\).

We now study specializations of  modified normal bundles when 
 we peel off a \(1\)-secant line. Let \(\mathcal{C} \subset \Gr \times \Delta \to \Delta\) be a family of curves
with smooth total space, whose central fiber $\mathcal{C}_0$ is a reducible curve \(C \cup_x L \subset \Gr\), where $C$ and $L$ meet quasi-transversely at a single point $x \in \Gr$.
Write \(\mathcal{N}\) for the normal bundle of \(\mathcal{C}\) in \(\Gr \times \Delta\);
the fiber of \(\mathcal{N}\) over any \(t \in \Delta\) is $N_{\mathcal{C}_t/\Gr}$.  We consider a modification \(\mathcal{N}'\) of \(\mathcal{N}\) along a collection of Cartier divisors $\{D_i\}$ that do not meet \(L\). The general fiber $\mathcal{N}'_t$ for $t \in \Delta^*$ is a modification of $N_{\mathcal{C}_t/\Gr}$. The central fiber \(\mathcal{N}'_0 \) 
is a modification \(N'_{C \cup_x L / \Gr}\) of the normal bundle \(N_{C \cup_x L / \Gr}\).

We write \(N'_C\) for the vector bundle obtained by applying the corresponding modifications to \(N_C\), as we now make precise.
Since the divisors $D_i$ are disjoint from $L$, there exists an open set $U \subset C \cup L$ containing $L$, and an isomorphism \(N'_{C \cup_x L / \Gr}|_U \simeq N_{C \cup_x L / \Gr}|_U\).
To form \(N_C'\), we glue \(N'_{C \cup_x L / \Gr}|_{C \smallsetminus x}\) to \(N_C|_U\) via the isomorphism
\(N'_{C \cup_x L / \Gr}|_{U \smallsetminus x} \simeq N_{C \cup_x L / \Gr}|_{U \smallsetminus x} \simeq N_C|_{U \smallsetminus x}\).

Since $L$ is a $(-1)$-curve in the fiber of $\mathcal{C} \to \Delta$, we can blow $L$ down to obtain a new family $\mathcal{C}^- \to \Delta$ whose central fiber is \(C\). 

\begin{prop} \label{prop:degen}
Let \(p\) and \(h\) be a point and a hyperplane in \(\pp V\) respectively, such that every \(a\)-plane in \(L\) contains \(\omega = x \cap h\) and is contained in \(\gamma = x + p\).
There exists a vector bundle on \(\mathcal{C}^-\) whose general fiber agrees with \(\mathcal{N}'|_{\Delta^*}\) and whose special fiber is
\[N'_C[x \arrowdown p][x \arrowup h].\]
\end{prop}

\begin{proof}
Recall the notation \(\Omega = \Omega^h\) and \(\Gamma = \Gamma_p\) and \(\Sigma = \Sigma^h_p\) from Section \ref{sec:generalized_pointing}.
Since $\Sigma$ is a scroll swept out by nonintersecting lines and $L$ is one of these lines, we have that \(N_{L/\Sigma} \simeq \O_L\).
Moreover, \(\Sigma\) is transverse to \(\Omega(\omega)\) and \(\Gamma(\gamma)\) by generality,
and is contained in \(\Omega\) and \(\Gamma\) by construction.
These varieties are thus compatible with the following direct sum decomposition of \(N_{L/\Gr}\):

\begin{minipage}{0.95\textwidth}
\[\phantom{N_{L/\Gr} \simeq \underbrace{\O_{\pp^1}(1)^{b-1}}_{N_{L/\Omega(\omega)}} \oplus} \overbrace{\phantom{\underbrace{\O_{\pp^1}}_{N_{L/\Sigma}} \oplus \underbrace{\O_{\pp^1}(1)^{a - 1}}_{N_{L/\Gamma(\gamma)}}}}^{N_{L/\Gamma}} \phantom{\oplus \, \O_{\pp^1}^{ab - a - b}.}\]
\vspace{-50pt}
\[\phantom{N_{L/\Gr} \simeq} \overbrace{\phantom{\underbrace{\O_{\pp^1}(1)^{b-1}}_{N_{L/\Omega(\omega)}} \oplus \underbrace{\O_{\pp^1}}_{N_{L/\Sigma}}}}^{N_{L/\Omega}} \phantom{\oplus \underbrace{\O_{\pp^1}(1)^{a - 1}}_{N_{L/\Gamma(\gamma)}} \oplus \, \O_{\pp^1}^{ab - a - b}.}\]
\vspace{-40pt}
\[N_{L/\Gr} \simeq \underbrace{\O_{\pp^1}(1)^{b-1}}_{N_{L/\Omega(\omega)}} \oplus \underbrace{\O_{\pp^1}}_{N_{L/\Sigma}} \oplus \underbrace{\O_{\pp^1}(1)^{a - 1}}_{N_{L/\Gamma(\gamma)}} \oplus \, \O_{\pp^1}^{ab - a - b}.\]
\end{minipage}

\noindent
Moreover, since \(\Sigma\) is a smooth surface containing \(C \cup L\),
\[ N'_{C \cup L / \Gr}|_L = N_{C \cup L / \Gr}|_L \simeq N_{L/\Gr}[x \to N_{L/\Sigma}] \quad \text{and} \quad N'_{C \cup L / \Gr}|_C \simeq N'_{C/\Gr}[x \to N_{C/\Sigma}].\]
In particular, 
\[ N'_{C \cup L / \Gr}|_L \simeq \O_{\pp^1}(1)^{a+b-1} \oplus \O_{\pp^1}^{ab-a-b},\]
where the positive subbundle $P= \O_{\pp^1}(1)^{a+b-1}$ corresponds to $(N_{L/\Gamma} + N_{L/\Omega})[x \to N_{L/\Sigma}]$.  In particular, over \(x = C \cap L\), the fiber of \(P\) glues to the fiber of \((N_{C/\Gamma} + N_{C/\Omega})[x \to N_{C/\Sigma}]\).

The line \(L\) is a Cartier divisor on \(\mathcal{C}\), so we can construct the modification \(\mathcal{N'}[L \to P]\).
By construction,
\[\mathcal{N}'[L \to P]|_L \simeq \O_{\pp^1}^{ab-1},\]
so $\mathcal{N}' [L \to P]$ is the pullback of a vector bundle $\mathcal{N}^-$ from $\mathcal{C}^-$.
Moreover, since \(L\) is supported on the central fiber $\mathcal{C}_0$,
\[\mathcal{N}^-|_{\Delta^*} \simeq \mathcal{N}'|_{\Delta^*}.\]
To complete the proof, we compute the central fiber of \(\mathcal{N}^-\) as follows:
\begin{align*}
\mathcal{N}^-_0 &\simeq \mathcal{N}'[L \to P]|_C \\
&\simeq N'_{C/\Gr}[x \to N_{C/\Sigma}][x \to N_{C / \Gamma} + N_{C / \Omega}] \\
&\simeq N'_{C/\Gr}[x \to N_{C/\Gamma}][x \to N_{C/\Omega}] \\
&\simeq N'_C[x \arrowdown p][x \arrowup h]. \qedhere\end{align*}
\end{proof}

We next explain how to modify the above analysis to handle the case \(a = 1\).
In this case \(\omega = 0\), and so there is no need to introduce \(h\), \(\Omega(\omega)\), or \(\Omega\);
we also have \(\Gamma(\gamma) = L\) and \(\Gamma = \Sigma\).

\begin{prop} \label{prop:degen-1}
Let \(p\) be a point in \(\pp V\), such that \(L = \overline{xp}\).
There exists a vector bundle on \(\mathcal{C}^-\) whose general fiber agrees with \(\mathcal{N}'|_{\Delta^*}\) and whose special fiber is
\[N'_C(x)[2x \arrowdown p].\]
\end{prop}
\begin{proof}
The key difference from the proof Proposition~\ref{prop:degen}
is that the lines sweeping out \(\Sigma\) are no longer nonintersecting,
since they all pass through the common point \(p\).
Thus we instead have \(N_{L/\Sigma} \simeq \O_L(1)\),
which is compatible with a direct sum decomposition of \(N_{L/\pp V}\) of the form:
\[N_{L/\pp V} \simeq \O_{\pp^1}(1)^{b - 2} \oplus \underbrace{\O_{\pp^1}(1)}_{N_{L/\Sigma}}.\]
A similar analysis as in the proof of Proposition~\ref{prop:degen} therefore yields
\[ N'_{C \cup L / \pp V}|_L \simeq \O_{\pp^1}(1)^{b - 2} \oplus \O_{\pp^1}(2),\]
where the positive subbundle $P= \O_{\pp^1}(2)$ corresponds to $N_{L/\Sigma}(x)$.
In particular, over \(x = C \cap L\), the fiber of \(P\) glues to the fiber of \(N_{C/\Sigma}(x)\).

As in the proof Proposition~\ref{prop:degen},
the modification $\mathcal{N}' (L)[L \to P]$ is the pullback of a vector bundle $\mathcal{N}^-$ from $\mathcal{C}^-$.
To complete the proof, we compute the central fiber of \(\mathcal{N}^-\) as follows:
\begin{align*}
\mathcal{N}^-_0 &\simeq \mathcal{N}'(L)[L \to P]|_C \\
&\simeq N'_{C/\pp V}[x \to N_{C/\Sigma}](x)[x \to N_{C / \Sigma}] \\
&\simeq N'_{C/\pp V}(x)[2x \to N_{C/\Sigma}] \\
&\simeq N'_C(x)[2x \arrowdown p]. \qedhere\end{align*}
\end{proof}

As an application of Proposition~\ref{prop:degen-1},
we prove the following.

\begin{prop}\label{prop:char2}
Let \(C \subset \pp^b\) be a general rational curve of degree \(d\).
Then \(N_C\) is \(2\)-balanced.
\end{prop}

\begin{rem}
In characteristic \(2\),
we have shown in Section~\ref{ss:c2} that every summand of \(N_C\) is congruent to \(d\) mod \(2\).
Since there is a unique \(2\)-balanced bundle of degree \((b + 1)d - 2\) and rank \(b - 1\) with this property,
we can describe \(N_C\) more explicitly as follows.
Write \(d - 1 = (b - 1)q + r\) where \(0 \leq r < b - 1\). Then
\[N_{C / \pp^b} \simeq \O_{\pp^1}(d + 2q + 2)^{\oplus r} \oplus \O_{\pp^1}(d + 2q)^{\oplus (b - 1 - r)}.\]
\end{rem}

\begin{proof}[Proof of Proposition~\ref{prop:char2}]
When the characteristic is not \(2\),
it is known that the normal bundle of a general nondegenerate rational curve \(C \subset \pp^b\) is balanced
\cite{sacchiero80, interpolation}.
If \(C\) is instead degenerate, then by writing \(\pp^d\) for the linear span of \(C\) and considering the exact sequence
\[0 \to \O_{\pp^1}(d + 2)^{d - 1} \simeq N_{C / \pp^d} \to N_{C / \pp^b} \to N_{\pp^d/\pp^b}|_C \simeq \O_{\pp^1}(d)^{b - d} \to 0,\]
we conclude that the normal bundle of any rational curve (even degenerate) is \(2\)-balanced in characteristic not \(2\).

Now assume that the characteristic is \(2\).  We will prove that the normal bundle is \(2\)-balanced by induction on \(d\).
The base case \(d = 1\) holds since the normal bundle of a line in \(\pp^b\) is \(\O_{\pp^1}(1)^{b - 1}\),
which is in particular \(2\)-balanced.

For the inductive step, we degenerate to \(C \cup_x L \subset \pp V\)
and apply Proposition~\ref{prop:degen-1}.
This induces a specialization of the normal bundle of a general rational curve of degree \(d + 1\) in \(\pp^b\) to
\(N_C(x)[2x \arrowdown p]\).

By our inductive hypothesis, \(N_C\), and thus \(N_C(x)\), is \(2\)-balanced.
Since \(p\) is general, the fiber of \(N_{C \to p}|_x\) is general,
so the single modification \([x \arrowdown p]\)
increases the degree of the smallest summand by one.
In particular \(N_C(x)[x \arrowdown p]\) is also \(2\)-balanced.
Making another modification at \(p\) increases the degree of some summand by one,
so \(N_C(x)[2x \arrowdown p]\) must be \(3\)-balanced.

From Section~\ref{ss:c2}, the normal bundle of a general rational curve of degree \(d + 1\) in \(\pp^b\)
has every summand congruent to \(d + 1\) mod \(2\).
Since a \(3\)-balanced bundle with this property must be \(2\)-balanced, completing the inductive step.
\end{proof}

\section{The proof of the main theorem}\label{sec-proof}
In this section, we prove Theorem~\ref{thm-main} by induction. 
Theorem~\ref{thm-main} is the \(n = 0\) case of the following inductive hypothesis.

\begin{thm}\label{thm:stronger}
Let \(C \subset \Gr(a, a + b)\) be a general rational curve of degree \(d \geq 1\) in a Grassmannian \(\Gr(a, a + b)\)
with \(a \geq 1\) and \(b \geq 2\), and let \(n \geq 0\) be a nonnegative integer.
For general points \(x_i \in C\) and \(p_i \in \pp^{a + b - 1}\), let
\[N' \colonequals N_C[x_1 \arrowdown p_1] \cdots [x_n \arrowdown p_n].\]
Then either
\begin{enumerate}
\item\label{part:2bal} The bundle \(N'\) is \(2\)-balanced; or
\item\label{part:bounded} The bundle \(N' \simeq \bigoplus \O(a_i)\) with each \(a_i \leq \left\lceil \frac{d}{a} + n \right\rceil\), and
\(\frac{d}{a} + n > \frac{(a + b)d - 2 + an}{ab - 1}\).
\end{enumerate}
\end{thm}

\begin{rem}
    Observe 
\(\mu(N') = \frac{(a + b)d - 2 + an}{ab - 1}\).
\end{rem}

\begin{rem}\label{rem-n=0}
When $n=0$, we have \(\frac{d}{a}\leq \frac{(a + b)d - 2}{ab - 1}\). Hence, Theorem \ref{thm:stronger} implies that \(N_C\) is \(2\)-balanced and  Theorem \ref{thm-main} holds.
\end{rem}

We will now prove Theorem \ref{thm:stronger} by induction. Our induction will have two base cases:
First, if \((d, n) = (1, 0)\), then \(N'\) is the normal bundle of a line in \(\Gr(a, a + b)\),
hence balanced. We therefore suppose for the remainder of this section that \((d, n) \neq (1, 0)\).
Second, if \(a = 1\), then 
by Proposition~\ref{prop:char2}, the normal bundle of a general rational curve in projective space is \(2\)-balanced.  Since
\(N'\) is obtained by making
\(n\) general positive modifications to
the normal bundle of a rational curve in projective space,
and so \(N'\) is $2$-balanced.
We therefore suppose \(a, b \geq 2\). 
We remark that the induction will proceed from \((a, b)\) to
either \((a - 1, b)\) or \((b - 1, a)\).

We will split the inductive step of the proof of Theorem~\ref{thm:stronger} into three cases, based on the size of \(\frac{d}{a} + n\).

\begin{proof}[Proof of Theorem~\ref{thm:stronger} when \(\frac{d}{a} + n > \frac{(a + b)d - 2 + an}{ab - 1}\)]
We specialize all \(p_i\) to a common point \(p\), which induces a specialization of \(N'\) to
\(N_C[x_1 + \cdots + x_n \arrowdown p]\).
The exact sequence \eqref{eq:proj_p} for projection from \(p\) induces the exact sequence
\[0 \to S|_C^\vee(n) \simeq N_{C \to p}(x_1 + \cdots + x_n) \to N_C[x_1 + \cdots + x_n \arrowdown p] \to N_{\pi_p(C)} \to 0.\]
The subbundle \(S|_C^\vee(n)\) is balanced of slope \(\frac{d}{a} + n\), since \(S|_C^\vee\) is balanced for general \(C\).
After applying duality, the quotient bundle \(N_{\pi_p(C)}\) is another instance of our inductive hypothesis with \((a', b') = (b-1,a)\)  and \(n=0\).  Therefore, by Remark \ref{rem-n=0} and induction, we have that
\(N_{\pi_p(C)}\) is \(2\)-balanced.
By Lemma~\ref{lem:ses_2bal}, either \(N_C[x_1 + \cdots + x_n \arrowdown p]\) is \(2\)-balanced, or all of the summands of \(N_{\pi_p(C)}\) have degree at most \(\left\lceil\frac{d}{a} + n\right\rceil\).  In the first case, Theorem~\ref{thm:stronger}\eqref{part:2bal} holds because \(2\)-balancedness is open.  In the second case, since all the summands of \(S|_C^\vee(n)\) also have degree at most \(\left\lceil\frac{d}{a} + n\right\rceil\), we conclude by Lemma~\ref{lem:interval} that the same is true of \(N_C[x_1 + \cdots + x_n \arrowdown p]\), and hence of \(N'\), since this is an open condition.  Thus Theorem~\ref{thm:stronger}\eqref{part:bounded} holds.
\end{proof}

\noindent
In the next two cases, we will need the following numerical lemma.

\begin{lem} \label{lem:dm1}
Let $a,b,d, n$ be integers with $a,b \geq 2$ and $d \geq 1$ and $n \geq 0$. Assume that  \((d, n) \neq (1, 0)\). Then
\[\frac{(a + b)d - 2 + an}{ab - 1} - \frac{d}{a} - n \leq d - 1 \quad \text{and} \quad \frac{d + n}{b} < \frac{(a + b)d - 2 + an}{ab - 1}.\]
\end{lem}
\begin{proof}
Since \((d, n) \neq (1, 0)\), we have \(d + n \geq 2\).
Upon rearrangement, the first inequality becomes
\[[a^2(b - 2) + (a + 1)(a - 2) + 1](d + n - 2) + n + a^2(b - 2) + (a - 2) \geq 0,\]
which holds by our assumptions.  Similarly, the second inequality becomes
\begin{equation*}
(b^2 + 1)(d - 1) + (b - 1)^2 + n > 0. \qedhere
\end{equation*}
\end{proof}

\begin{proof}[Proof of Theorem~\ref{thm:stronger} when \(\frac{d + n}{b} \leq \frac{d}{a} + n \leq \frac{(a + b)d - 2 + an}{ab - 1}\)]
Subject to these inequalities, Theorem~\ref{thm:stronger}\eqref{part:bounded} cannot hold, and hence we want to prove that \(N'\) is \(2\)-balanced.
We will do this by constructing two specializations of \(N'\), one providing an upper bound on the degrees of summands of \(N'\), and the other providing a lower bound.
Define
\begin{equation}\label{eqn-defdelta}
    \delta = \frac{(a + b)d - 2 + an}{ab - 1} - \frac{d}{a} - n.
\end{equation}

By assumption and Lemma~\ref{lem:dm1}, we have \(0 \leq \delta \leq d - 1\).
We iteratively specialize \(C\) to the union of a rational curve of degree one less with a \(1\)-secant line,
applying Proposition~\ref{prop:degen} each time.
We will show that:
\begin{enumerate}
\item\label{case1} (Applying Proposition~\ref{prop:degen} \(\lfloor \delta \rfloor\) times) If \(C\) has degree \(d - \lfloor \delta \rfloor\), then
\[N_C[x_1 \arrowdown p_1] \cdots [x_n \arrowdown p_n][y_1 \arrowdown q_1][y_1 \arrowup h_1] \cdots [y_{\lfloor \delta \rfloor} \arrowdown q_{\lfloor \delta \rfloor}][y_{\lfloor \delta \rfloor} \arrowup h_{\lfloor \delta \rfloor}]\]
is either \(2\)-balanced or has all summands at least \(\lfloor \frac{d}{a} + n \rfloor + \lfloor \delta \rfloor\); and

\item\label{case2} (Applying Proposition~\ref{prop:degen} \(\lceil \delta \rceil\) times) If \(C\) has degree \(d - \lceil \delta \rceil\), then
\[N_C[x_1 \arrowdown p_1] \cdots [x_n \arrowdown p_n][y_1 \arrowdown q_1][y_1 \arrowup h_1] \cdots [y_{\lceil \delta \rceil} \arrowdown q_{\lceil \delta \rceil}][y_{\lceil \delta \rceil} \arrowup h_{\lceil \delta \rceil}]\]
is either \(2\)-balanced or has all summands at most \(\lceil \frac{d}{a} + n \rceil + \lceil \delta \rceil\).
\end{enumerate}
This suffices to complete the proof, since either one of these specializations implies that \(N'\) is \(2\)-balanced, or together, they imply that all of the summands of \(N'\) lie in the interval 
\[\left[\left\lfloor \frac{d}{a} + n \right\rfloor + \left\lfloor \delta \right\rfloor, \left\lceil \frac{d}{a} + n \right\rceil + \left\lceil \delta \right\rceil \right],\] 
which has length at most \(2\).

We prove both of these the same way: Specialize all \(p_i\) and \(q_i\) together to a common point \(p\) and specialize all \(h_j\) to contain \(p\), which induces a specialization of \(N'\) to \(N'_1\),
and use the exact sequence \eqref{eq:proj_p} of projection from \(p\):
\[0 \to S' \to N'_1 \to N_{\pi_p(C)}[y_1 \arrowup \pi_p(h_1)] \cdots [y_{[\delta]} \arrowup \pi_p(h_{[\delta]})] \to 0,\]
where
\[S' \colonequals S|^\vee_C(n + [\delta])[y_1 \arrowup \pi_p(h_1)] \cdots [y_{[\delta]} \arrowup \pi_p(h_{[\delta]})],\]
and where \([\delta]\) is either the floor or ceiling depending on if we are in Case~\eqref{case1} or \eqref{case2}.
Applying duality, the quotient is another instance of our inductive hypothesis with \((a', b') = (b-1, a)\).  In Case~\eqref{case1}, we will verify the inequality
\begin{equation}\label{ineq:floor}
    \frac{d - \lfloor \delta \rfloor}{b - 1} + \lfloor \delta \rfloor \leq \frac{(a + b - 1)(d - \lfloor \delta \rfloor) - 2 + (b - 1)\lfloor \delta \rfloor}{a(b - 1) - 1},
\end{equation} 
which guarantees that the quotient is \(2\)-balanced
when we apply our inductive hypothesis.  The inequality assumed in this case implies that \(\mu(S') \leq \mu(N'_1)\). 
Since \(S'\) is balanced by Corollary~\ref{lem:gen_mods_S}, Lemma~\ref{lem:ses_2bal} implies that either \(N'_1\) (and hence \(N'\)) is \(2\)-balanced, or all summands of \(N'_1\) (and hence of \(N'\)) have degree at least \(\left\lfloor\mu(S') \right\rfloor = \lfloor \frac{d}{a} + n \rfloor + \lfloor \delta \rfloor\).

In Case~\eqref{case2}, we verify the inequality
\begin{equation}\label{ineq:ceil}
    \left\lceil \frac{d - \lceil \delta \rceil}{b - 1} + \lceil \delta \rceil \right\rceil \leq \left\lceil \frac{d}{a} + n \right\rceil + \lceil \delta \rceil,
\end{equation} 
which guarantees that 
either the quotient \(N_{\pi_p(C)}[y_1 \arrowup \pi_p(h_1)] \cdots [y_{\lceil\delta\rceil} \arrowup \pi_p(h_{\lceil\delta\rceil})]\) is \(2\)-balanced, or all of its summands have degree at most \(\left\lceil \frac{d}{a} + n \right\rceil + \lceil \delta \rceil\).  In the first case, we conclude as above, combining the fact that  \(S'\) is balanced and \(\mu(S') \geq \mu(N'_1)\) in Lemma~\ref{lem:ses_2bal}.  In the second case, since the summands of \(S'\) are also bounded by \(\left\lceil \frac{d}{a} + n \right\rceil + \lceil \delta \rceil\), we conclude that the same is true of \(N'_1\) by Lemma~\ref{lem:interval}.

We now verify the inequalities \eqref{ineq:floor} and \eqref{ineq:ceil}.
Upon rearranging, these two inequalities  follow from:
\begin{equation} \label{eq:delta1}
\frac{d +  (b-2)\lfloor \delta \rfloor}{b - 1}  \leq \frac{(a + b - 1)d - a \lfloor \delta \rfloor - 2 }{a(b - 1) - 1} \quad \text{and} \quad \frac{d - \lceil \delta \rceil}{b - 1} \leq \frac{d}{a} + n.
\end{equation}
In fact, we will show that the following stronger inequalities hold:
\begin{equation} \label{eq:delta1-nfc}
\frac{d +  (b-2) \delta }{b - 1}  \leq \frac{(a + b - 1)d - a  \delta  - 2 }{a(b - 1) - 1}
 \quad \text{and} \quad \frac{d - \delta}{b - 1} \leq \frac{d}{a} + n,
\end{equation}
except in the following cases:
\begin{enumerate}
\item \label{e1} If \(n = 0\) and \(b = a\) and \(d \leq a - 1\);
\item \label{e2} If \(n = 1\) and \(b \leq (a + 1)/3\) and \(d = a(b - 1)/(a - b)\).
\end{enumerate}
Observe that the two inequalities \eqref{eq:delta1-nfc} are equivalent:
In fact, upon rearrangement and substituting the definition of \(\delta\) from \eqref{eqn-defdelta}, both can be written as:
\begin{equation}\label{eq:single_ineq}
    [(b - a)(ab - a - 1) + 2]d + (a^2 (b-1)^2 - ab + 2a)n \geq 2a.
\end{equation}
To verify that \eqref{eq:single_ineq} holds except in cases \eqref{e1} and \eqref{e2}, we divide into cases based on the relative sizes of \(a\) and \(b\).

If \(b > a \geq 2\), since $(b-1)^2 >b$,  the coefficients of both \(d\) and \(n\) are 
at least \(2a\). Since \(d \geq 1\) and \(n \geq 0\), the inequality follows.

If \(b = a \geq 2\), then the coefficient of \(n\) is at least \(2a\),
and the coefficient of \(d\) is \(2\).
Thus the inequality holds unless \(n = 0\) and \(d < a\), which is case~\eqref{e1} above.

If \(b < a\), then the left-hand side of \eqref{eq:single_ineq} is decreasing in \(d\).
From our assumption that \(\frac{d + n}{b} \leq \frac{d}{a} + n\), we conclude \(n \geq 1\) and
\(d \leq (abn - an) / (a - b)\).
If this bound on \(d\) is strict, then \(d \leq (abn - an - 1) / (a - b)\).
Substituting this value into \eqref{eq:single_ineq} and rearranging, it suffices to check
\[(a^2b + ab^2 - 2ab)(n - 1) + (a^2b^2 - ab^3 - a^2b + 3ab^2 - 3ab + b^2 - a - b) \geq 0.\]
Observing that 
$$a^2b^2 - ab^3 - a^2b + 3ab^2 - 3ab + b^2 - a - b = (a-b-1)^2(b^2-b) + (a-b-1)(b^3 +3b^2-5b-1) + (3b^3-6b-1),$$
we see that the inequality holds for \(n \geq 1\) and \(a > b \geq 2\).
We are thus reduced to the case \(d = (abn - an) / (a - b)\).
Substituting this value in \eqref{eq:single_ineq}, our inequality reduces to
\[(a + b - 2)n \geq 2(a - b),\] which holds for \(n \geq 2\).
Our inequality therefore holds unless \(n = 1\) and
\(a + b - 2 < 2(a - b)\), or upon rearrangement, \(b < (a + 2)/3\).
We thus fall into case~\eqref{e2} above.

To complete the proof, we must establish the inequalities \eqref{eq:delta1} under the hypotheses
\eqref{e1} or \eqref{e2}.
Note that both of these cases imply equality in our assumption: \(\frac{d + n}{b} = \frac{d}{a} + n\). In particular \(\delta > 0\)
by Lemma~\ref{lem:dm1}.
We first claim that both of these sets of hypotheses imply \(\delta < 1\).
Indeed, for \eqref{e1}, we substitute \(n = 0\) and \(b = a\) into \(\delta\)
to obtain \(\delta = ((a^2 + 1)d - 2a)/(a^3 - a)\). This is increasing in \(d\),
thus maximized when \(d = a - 1\), for which it is \((a^3 - a^2 - a - 1)/(a^3 - a) < 1\).
For \eqref{e2}, we substitute \(n = 1\) and \(d = a(b - 1)/(a - b)\) into \(\delta\)
to obtain \(\delta = (ab^2 - ab - a + 2b - 1)/(a^2b - ab^2 - a + b)\).
It thus suffices to show \((ab^2 - ab - a + 2b - 1)/(a^2b - ab^2 - a + b) < 1\),
or upon rearrangement, \((a^2 + a - 1 - 2ab)b > -1\), which follows from the nonnegativity of \(a^2 + a - 1 - 2ab\), which we now show.
This quantity is decreasing in \(b\), thus minimized when \(b = (a + 1)/3\),
for which it is \((a^2 + a - 3)/3 \geq 0\).

The upshot is that, in both of these cases, \(\lfloor \delta \rfloor = 0\) and \(\lceil \delta \rceil = 1\).
Substituting these values into our original inequalities \eqref{eq:delta1} and simplifying, it remains to check
\[(b^2 - 2b + 2)(d - 1) + (b - 2)^2 \geq 0 \quad \text{and} \quad (b - a - 1)d + (ab - a)n + a \geq 0.\]
The first of these inequalities always holds, and the second evidently holds in case \eqref{e1},
so it remains to check the second in case \eqref{e2}.
Substituting in \(n = 1\), and \(d = a(b - 1)/(a - b)\), it rearranges to
\(b \leq (a + 1)/2\), which is  implied by \(b \leq (a + 1)/3\).
\end{proof}

\begin{proof}[Proof of Theorem~\ref{thm:stronger} when \(\frac{d}{a} + n < \frac{d + n}{b} < \frac{(a + b)d - 2 + an}{ab - 1}\)]
Define
\[\epsilon = \frac{(a + b)d - 2 + an}{ab - 1} - \frac{d + n}{b}.\]

Our argument proceeds as in the previous case, peeling off both \(\lfloor\epsilon\rfloor\) and \(\lceil\epsilon \rceil\) lines and applying Proposition~\ref{prop:degen}.
Instead of specializing the \(p_i\) and \(q_i\) together and projecting from \(p\), we 
specialize all of the \(h_j\) together to a common hyperplane \(h\), specialize all of the \(p_i\) to lie in \(h\), 
and use the exact sequence \eqref{eq:int_h} of intersection with \(h\).
This reduces us directly to an instance of our inductive hypothesis, without applying duality. Analogously to the previous case, it suffices to check the inequalities:
\[\frac{d - \lfloor \epsilon \rfloor}{a - 1} + n + \lfloor \epsilon \rfloor \leq \frac{(a + b - 1)(d - \lfloor \epsilon \rfloor) - 2 + (a - 1)(n + \lfloor \epsilon \rfloor)}{(a - 1)b - 1} \quad \text{and} \quad \left\lceil \frac{d - \lceil \epsilon \rceil}{a - 1} + n + \lceil \epsilon \rceil \right\rceil \leq \left\lceil \frac{d + n}{b}\right\rceil + \lceil \epsilon \rceil.\]
Upon rearrangement, these inequalities follow from:
\begin{equation} \label{eq:epsilon2}
\frac{d + (a - 2)\lfloor \epsilon \rfloor}{a - 1} + n \leq \frac{(a + b - 1)d + (a - 1)n - b \lfloor \epsilon \rfloor - 2}{(a - 1)b - 1} \quad \text{and} \quad \frac{d - \lceil \epsilon \rceil}{a - 1} + n \leq \frac{d + n}{b}.
\end{equation}
In fact, we claim that the following stronger inequalities hold, this time with no exceptions:
\begin{equation} \label{eq:epsilon2-nfc}
\frac{d + (a - 2) \epsilon}{a - 1} + n \leq \frac{(a + b - 1)d + (a - 1)n - b \epsilon - 2}{(a - 1)b - 1} \quad \text{and} \quad \frac{d - \epsilon}{a - 1} + n \leq \frac{d + n}{b}.
\end{equation}
To see this claim, we first observe that the two inequalities \eqref{eq:epsilon2-nfc} are equivalent. In fact, upon rearrangement,
both can be written as:
\begin{equation}\label{ineq:dpos}
[(a - b)(ab - a - 1) + 2]d - (a^2b^2 - a^2b - ab^2 + a + b - 2)n \geq 2b.
\end{equation}
This time our assumption \(\frac{d}{a} + n < \frac{d + n}{b}\) forces \(a > b\),
so the coefficient of \(d\) in \eqref{ineq:dpos} is positive.

If \(n = 0\), then substituting \(n = 0\) and \(d \geq 1\),
it suffices to check \((a - b)(ab - a - 1) + 2 \geq 2b\),
which holds for \(a > b \geq 2\).

Otherwise, our assumption \(\frac{d}{a} + n < \frac{d + n}{b}\) rearranges to
\(d > (abn - an)/(a - b)\), or equivalently, \(d \geq (abn - an + 1)/(a - b)\).
Substituting this into our inequality and rearranging, it suffices to check
\[(ab + b^2 - 2b)(n - 1) + a^2b - ab^2 - 2ab + 4b^2 - a - b + 2 \geq 0,\]
which holds for \(n \geq 1\) and \(a > b \geq 2\) because
\[a^2b - ab^2 - 2ab + 4b^2 - a - b + 2 = (a - b - 1)^2 b + (a - b - 1)(b^2 - 1) + (3b^2 - 3b + 1). \qedhere\]
\end{proof}

\bibliographystyle{plain}

\end{document}